\documentclass{article}
\title{Existence Results and KKT Optimality Conditions for Generalized Quasiconvex Functions}
\author{M.H. Alizadeh\thanks{Department of Mathematics, Institute for Advanced Studies in Basic Sciences (IASBS), Zanjan 45137-66731, Iran.
E-mail:m.alizadeh@iasbs.ac.ir; alimh.math@gmail.com}
\and
F. Lara\thanks{Instituto de Alta investigaci\'on (IAI),
Universidad de Tarapac\'a, Arica, Chile. E-mail:
felipelaraobreque@gmail.com; flarao@academicos.uta.cl. Web:
felipelara.cl, ORCID-ID: 0000-0002-9965-0921}}
\usepackage{amsmath}
\usepackage{amsfonts}
\usepackage{amssymb}
\usepackage{graphicx}%
\usepackage{color}
\usepackage{cite}

\providecommand{\U}[1]{\protect \rule{.1in}{.1in}}
\newtheorem{theorem}{Theorem}

\newtheorem{corollary}[theorem]{Corollary}

\newtheorem{definition}[theorem]{Definition}
\newtheorem{example}[theorem]{Example}

\newtheorem{proposition}[theorem]{Proposition}
\newtheorem{remark}[theorem]{Remark}

\newenvironment{proof}[1][Proof]{\noindent \textbf{#1.} }{\  \rule{0.5em}{0.5em}}
\begin{document}

\maketitle

\begin{abstract}
\noindent 
We studied a new notion of generalized convex functions called $e$-quasi\-con\-ve\-xi\-ty, which encompasses both quasiconvex and $e$-convex functions, including all Lipschitz functions. By extending the standard properties of quasiconvex functions to $e$-quasiconvex functions, we establish sufficient conditions for the nonemptiness and compactness of the solution set when minimizing an $e$-quasiconvex function, leveraging generalized asymptotic functions, a result which remains applicable even when the set of minimizers is nonconvex.
Furthermore, in the differentiable case, we ensure the sufficiency of the KKT optimality conditions when the constraint functions in the mathematical programming problems are $e$-quasiconvex. Finally, we illustrate our new results with several nonconvex (non-quasiconvex) examples.

\medskip


\medskip


\end{abstract}


\section{Introduction}

Convex functions play a crucial role in applied mathematics and optimization, particularly in fields such as engineering, management, computer science, and machine learning. They provide fundamental insights and powerful tools for research in these disciplines due to their very interesting properties, in which local analysis of convex functions often leads to global conclusions, an extremely useful feature for proving the existence of solutions, deriving optimality conditions, and ensuring the convergence of algorithms to global minimizers, among other applications. For a deeper exploration of convex functions, please refer to Rockafellar’s renowned book \cite{rock}.  

Due to these remarkable properties for convex functions, many researchers have sought to extend them to broader classes of functions. Although several generalized convexity notions have been introduced, one of the most significant and widely applied is the concept of {\it quasiconvexity}.

Quasiconvex functions are intrinsically linked to the well-known Consumer Preference Theory (see Debreu \cite{D-1959}), as they provide the mathematical formulation of the {\it tendency to diversification} assumption, a cornerstone in the analysis of consumer behavior in modern economic theory. Motivated primarily by this connection, researchers have extensively studied this class of functions from both theoretical and applied perspectives (see \cite{Ansari-book, AE, ADSZ, CMA, FFB-Vera, HKS, M1, Manga, Penot}, among others), extending many remarkable properties of convex functions to quasiconvex settings. Moreover, in recent years, algorithms tailored for quasiconvex functions have been developed across various domains, contributing significantly to advances in generalized convexity theory and, in particular, to the field of quasiconvex optimization.

Other notions of generalized convexity have also been developed due to their theoretical significance. Interesting examples include $\epsilon$-convex functions (see \cite{Jo-Luc-MI}), $\sigma$-convex functions (see \cite{Ali 2018, Ali 2021,Ali-abadi,Hui-Sun}), and $e$-convex functions (see \cite{ALi-Mako,Ali 203e-conv}). Notably, the class of $e$-convex functions is particularly broad, since it encompasses both $\epsilon$-convex and $\sigma$-convex functions, as well as all Lipschitz functions  (see \cite{Ali 2018, Ali 2021, Ali 203e-conv}). Moreover, this class admits a comprehensive theoretical framework, including the $e$-subdifferential, $e$-monotonicity, and its connections with the Fitzpatrick function (see \cite{ Ali 203e-conv, ABP, AZ 2025}).  

To simultaneously address both quasiconvex and $e$-convex functions, the authors of \cite{Huang 24} introduced the concept of $e$-quasiconvexity, establishing several important properties and optimization applications. In particular, they extended the celebrated Arrow-Enthoven (AE) characterization, originally developed for differentiable quasiconvex functions, to the $e$-quasiconvex differentiable case. This significant theoretical advancement has created new opportunities for applying $e$-quasiconvex functions in optimization problems, especially in nonconvex mathematical programming, which is the goal of this work.

This paper advances the theory of $e$-quasiconvexity and its optimization a\-ppli\-ca\-tions. Our main contributions include; A novel characterization of $e$-con\-vex functions as sums of $e$-quasiconvex functions and linear functions, extending the seminal result from \cite{ACJ 94} to this broader framework; Existence theorems and sufficient conditions for solution set compactness in constrained and un\-cons\-trai\-ned minimization pro\-blems involving $e$-quasiconvex functions, esta\-bli\-shed through generalized asymptotic functions (see \cite{FFB-Vera,flo-had-lara-2,HLM,HLL,lara,lara-lopez,Penot}). In particular, these results remain valid even for nonconvex minimizer sets. For differentiable cases, we leverage the characterization from \cite{Huang 24} to derive enhanced optimality conditions for nonconvex mathematical programming pro\-blems, thereby generalizing previous results on quasiconvexity from \cite{AE,M1,Manga} to the $e$-quasiconvex setting. Each theoretical advancement is substantiated with carefully constructed examples where our results apply uniquely, demonstrating cases where existing literature fails to provide solutions.

The paper is organized as follows: Section 2 is dedicated to introducing the notation and preliminary results on generalized convexity and asymptotic ana\-ly\-sis. In Section \ref{sec:03}, we study $e$-quasiconvex functions by extending the concepts of $e$-convex and quasiconvex functions. In Section \ref{sec:04}, we provide existence of solutions for minimizing $e$-quasiconvex functions as well as sufficient conditions for the compactness of the solution set. Finally, in Section \ref{sec:05}, we establish finer sufficient KKT optimality conditions for nonconvex mathematical programming problems in which the constraints are $e$-quasiconvex.

\section{Preliminaries and Basic Definitions}\label{sec:02}

Let $K$ be a nonempty set in $\mathbb{R}^{n}$, its closure is denoted by ${\rm cl}\,K$, its boundary by ${\rm bd}\,K$, its topological interior by ${\rm int}\,K$ and its convex hull by ${\rm co}\,K$. The open ball with center at $x$ and radius $\delta >0$ is denoted by $B(x, \delta)$. The sets $[0, + \infty[$ and $]0, + \infty[$ are denoted by $\mathbb{R}_{+}$ and by $\mathbb{R}_{++}$, respectively. Given a $d \in \mathbb{R}^{n}$, $\mathcal{L}(d)$ denotes the linear space generated by $d$.



Given any function $f:\mathbb{R}^{n} \rightarrow \overline{\mathbb{R}} := \mathbb{R} \cup \{ \pm \infty\}$, the effective domain of $f$ is defined by ${\rm dom}\,f:= \{x \in \mathbb{R}^n: f(x) < + \infty\}$. We say that $f$ is a proper function if $f(x)>-\infty$ for every 
$x \in \mathbb{R}^n$ and ${\rm dom}\,f$ is nonempty (clearly, $f(x) = + \infty$ for every $x \notin {\rm dom}\,f$). Throughout this paper, all functions are assumed to be proper. We denote by ${\rm epi}\,f := \{(x,t) \in \mathbb{R}^{n} \times \mathbb{R}: f(x) \leq t \}$ the epigraph of $f$, by $S_{\lambda} (f) := \{x \in \mathbb{R}^{n}: f(x) \leq \lambda\}$ 
the sublevel of $f$ at the height $\lambda \in \mathbb{R}$, respectively,  and by ${\rm argmin}_{\mathbb{R}^{n}}\,f$ we mean the set of all minimal points of $f$.

\subsection{Generalized Convexity}

A proper function $f: \mathbb{R}^{n} \rightarrow \overline{\mathbb{R}}$ with convex domain is said to be:
\begin{itemize}
 \item[$(a)$] convex if for every $x, y \in {\rm dom}\,f$, we have
 $$f(\lambda x + (1-\lambda)y) \leq \lambda f(x) +(1-\lambda) f(y), ~ \forall ~ \lambda \in 
 ]0, 1[,$$
 \item[$(b)$] quasiconvex if for every $x, y \in {\rm dom}\,f$, we have
 $$f(\lambda x + (1-\lambda) y) \leq \max \{f(x), f(y)\}, ~ \forall ~ \lambda \in [0,1].$$
\end{itemize}
We mention that every convex function is quasiconvex. The continuous function $f: \mathbb{R} \rightarrow \mathbb{R}$, with $f(x) := \min\{\lvert x \rvert, 1\}$, is quasiconvex without being convex. Furthermore, recall that $f$ is convex if and only if ${\rm epi}\,f$ is a convex set while $f$ is quasiconvex if and only if $S_{\lambda} (f)$ is a  convex set, for all $\lambda \in \mathbb{R}$.


Let $K \subseteq \mathbb{R}^{n}$ be an open convex set and $f: K \rightarrow \mathbb{R}$ be a differentiable function. Then $f$ is quasiconvex if and only if for every $x, y \in K$, we have
\begin{equation}\label{char:AE}
 f(x) \leq f(y) ~ \Longrightarrow ~ \langle \nabla f(y), x - y \rangle \leq 0;
\end{equation}

Let $f: \mathbb{R}^{n} \rightarrow \mathbb{R}$ be a differentiable function. 
Then $f$ is said to be pseudoconvex (see \cite{Manga}) if
\begin{equation}\label{pseudoconvex}
 \langle \nabla f(x), y - x \rangle \geq 0 ~ \Longrightarrow ~ f(y) \geq f(x).
\end{equation}
If $f$ is pseudoconvex, then every local minimum is global minimum 
\cite[Theorem 3.2.5]{CMA}.

Throughout this paper, $e: \mathbb{R}^{n} \times \mathbb{R}^{n} \rightarrow
\mathbb{R} \cup \{+ \infty\}$ is an error bifunction. We recall that a bifunction 
$e$ is an error bifunction if (see \cite {Ali 203e-conv})
\begin{itemize} 
 \item ${\rm dom}\,e := \{(x, y) \in \mathbb{R}^{n} \times \mathbb{R}^{n}: \, e(x,y) < + \infty\} \neq \emptyset$,

 \item $e$ is nonnegative, 

 \item $e$ is symmetric, i.e., for all $(x, y )\in {\rm dom}\,e$ we have $e(x, y) = e(y, x)$. 
\end{itemize}
The notion of $e$-convexity was introduced and studied in \cite{ALi-Mako}, where the relationships between Hermite-Hard type inequalities and $e$-convexity are provided. 

Using an error bifunction $e$, we can define $e$-convex functions as follows:

\begin{definition}
 Let $f: \mathbb{R}^{n} \rightarrow \mathbb{R} \cup \left \{ +\infty \right \}$ be a function and $e$ be an error bifunction such that ${\rm dom}\,f \times {\rm dom}\,f \subseteq {\rm dom} \,e$. We say that $f$ is $e$\emph{-convex} if
 \begin{equation}\label{convexineq}
  f (tx + (1-t)y) \leq tf(x) + (1-t) f(y) + t(1-t) e(x, y),
 \end{equation}
 for all $x, y \in \mathbb{R}^{n}$ and all $t \in \, ]0,1[$.
\end{definition}

Clearly, all convex functions are $e$-convex (with $e \equiv 0$), but the converse statement does not hold in general. Furthermore, all Lipschitz functions are 
$e$-convex (see \cite[Proposition 18]{Ali 2021}). For every function $f: \mathbb{R}^{n} \rightarrow \mathbb{R} \cup \left \{+ \infty \right\}$  the map $e_{f}: \mathbb{R}^{n} \times \mathbb{R}^{n} \rightarrow \mathbb{R}_{+} \cup \left\{ + \infty \right\}$ is defied as follows (see \cite{Ali 203e-conv}): if $x, y \in {\rm dom}\,f$, 
\[
e_{f}(x, y) = \inf\left\{ a \in \mathbb{R}_{+}: \, \frac{f(  tx + (1-t)y) - t f(x) - (1-t) f(y)}{t (1 - t)} \leq a, \, \forall \, t \in \, ]0, 1[ \right\},
\]
and otherwise $e_{f}(x, y) = + \infty$.

Note that $e_{f}$ is a nonnegative symmetric error bifunction such that $e_{f} 
(x, x) = 0$ for all $x \in {\rm dom}\,f$ and it represents the minimal error bifunction 
$e$ for which $f$ is $e$-convex (see \cite{Ali 203e-conv}). 

\subsection{Asymptotic Analysis}

Given a nonempty set $K$ in $\mathbb{R}^{n}$, we define the asymptotic cone 
of $K$ as
$$K^{\infty} := \left\{ u \in \mathbb{R}^{n}: ~ \exists ~ t_n \rightarrow + \infty, ~ \exists ~ x_n \in K \ \text{ such that } ~ \frac{x_n}{t_n} \rightarrow u \right\}.$$
In the case when $K$ is closed and convex, it is well known that 
(see \cite{AT})
\begin{equation}\label{eq:convexcone}
 K^{\infty} = \{ u \in \mathbb{R}^{n}: ~ x_0 + \lambda u \in K, ~ \forall ~ \lambda 
 \geq 0 \},
\end{equation}
where \eqref{eq:convexcone} does not depend on the choice of $x_{0} \in K$.

The asymptotic/recession function $f^{\infty}:\mathbb{R}^{n} \rightarrow 
\overline{\mathbb{R}}$ of a proper function $f$, is defined by 
$$ \mathrm{epi}\,f^{\infty} :=(\mathrm{epi}\,f)^{\infty}.$$
It follows that (see \cite[Theorem 2.5.1]{AT})
\begin{equation}\label{didieu}
 f^{\infty} (u) = \inf\left\{ \liminf_{k \rightarrow + \infty} \frac{f (t_k u_k)}{t_k}: ~ 
 t_k \to +\infty, ~ u_k \rightarrow u \right\}. 
\end{equation}
Moreover, when $f$ is lower semicontinuous (lsc henceforth) and convex, 
 \cite[Proposition 2.5.2]{AT} implies that
\begin{equation}\label{func:first}
 f^{\infty} (u) = \sup_{t > 0} \frac{f(x_{0} + tu) - f(x_{0})}{t} = \lim_{t \rightarrow 
 + \infty} \frac{f(x_{0} + tu) - f(x_{0})}{t},
\end{equation}
and all expressions in \eqref{func:first} do not depend on the choice of the point 
$x_{0} \in {\rm dom}\,f$.

A proper function $f$ is said to be coercive if $f(x) \rightarrow + \infty$ as 
$\lVert x \rVert \rightarrow + \infty$, or equivalently, if $S_{\lambda} (f)$ is 
bounded for all $\lambda \in \mathbb{R}$. Note that if $f^{\infty} (u) > 0$ for 
all $u \neq 0$, then $f$ is coercive. In addition, if $f$ is proper, lsc and convex, 
then $f$ is coercive if and only if (see \cite[Proposition 3.1.3]{AT})
\begin{equation}\label{char:convex}
 f^{\infty} (u) > 0,  ~ \forall ~ u \neq 0 ~ \Longleftrightarrow ~ {\rm argmin}_{\mathbb{R}^{n}}\,f \neq \emptyset {\rm ~ and ~ compact.}
\end{equation}

Relation \eqref{char:convex} shows the importance of the asymptotic function $f^{\infty}$ for dealing with convex functions. But for dealing with nonconvex 
functions, the asymptotic function $f^{\infty}$ is not good enough. Indeed, the continuous and quasiconvex function $f(x) = \sqrt{\lvert x \rvert}$, thus ${\rm argmin}_{\mathbb{R}}\,f = \{0\}$ (nonempty and compact), but $f^{\infty} \equiv 0$, i.e., $f^{\infty}$ does not provide any information on its set of minimizers.

For this reason, the authors in \cite{FFB-Vera,flo-had-lara-2,HLM,Iusem-Lara-5,lara-lopez,Penot} introduced and studied several generalized asymptotic functions for dealing with quasiconvex functions. In this direction, we can mention that the $q$-asymptotic function is one of the best asymptotic function for obtaining information from the infinity. Let us recall that the $q$-asymptotic function of any proper function $f$ is defined as 
\begin{equation}\label{asint:supq}
 f^{\infty}_{q} (u) := \sup_{x \in {\rm dom}\,f} \sup_{t > 0} \frac{f(x + t u) - f(x)}{t}.
\end{equation} 
Clearly, $f^{\infty}_{q} \geq f^{\infty}$, and equality holds when $f$ is convex. If 
$f$ is quasiconvex (resp. lsc), then $f^{\infty}_{q}$ is quasiconvex (resp. lsc).
Furthermore, given a lsc qua\-si\-con\-vex function $f$, \cite[Theorem 4.7]{FFB-Vera}
implies that
\begin{equation}\label{char:qasymptotic}
 f^{\infty}_{q} (u) > 0, ~ \forall ~ u \neq 0 ~ \Longleftrightarrow ~ {\rm argmin}_{ \mathbb{R}^{n}}\,f \neq \emptyset ~ {\rm and ~ compact}. 
\end{equation}
Note that $f^{\infty}_{q} (u) > 0$ for all $u \neq 0$ does not imply the coercivity 
of $f$ (see \cite[Exam\-ple 5.6]{FFB-Vera}), i.e., relation \eqref{char:qasymptotic} 
goes beyond convexity.

For a further study on $e$-convexity, generalized convexity, and generalized asymptotic functions we refer to \cite{Ali 2018, Ali 2021, Ali 203e-conv, ALi-Mako, Ansari-book, AE, ADSZ, CMA, FFB-Vera, HKS, Iusem-Lara-5,lara, Lara-9,lara-lopez,Penot,Soltan} and re\-fe\-ren\-ces therein

\section{Generalized Quasiconvex Functions}\label{sec:03}

In this section, we develop the usual properties of $e$-quasiconvex functions by extending the ones for $e$-convex and quasiconvex functions. 

\begin{definition}{\rm (\cite[Definition 3.1]{Huang 24})}\label{edf:e-qcx}
 Let $f: \mathbb{R}^{n} \rightarrow \mathbb{R} \cup \left \{+ \infty \right\}$ be a function and $e$ be an error bifunction such that ${\rm dom}\,f \times {\rm dom}\,f \subseteq {\rm dom}\,e$. We say that $f$ is $e$-quasiconvex if for all $x, y \in \mathbb{R}^{n}$ and all $t \in \, ]0, 1[$, we have
\begin{equation}\label{e:qcx}
 f(tx+ (1-t)y) \leq \max\{ f(x), f(y)\} + t(1-t) e(x, y). 
\end{equation}
\end{definition}
Clearly, every $e$-convex function is $e$-quasiconvex and every quasiconvex function is $e$-quasiconvex for every error function $e$, while the converse statements do not hold in general. As a consequence, and since every Lipschitz function is $e$-convex (see \cite[Proposition 18]{Ali 2021}), the class of $e$-quasiconvex functions is very large. 

Summarizing, 
 \begin{align*}
  \begin{array}{ccccccc}
  \, & \, & {\rm convex} & \Longrightarrow & {\rm quasiconvex}  \notag \\  
  \, & \, & \Downarrow & \, & \Downarrow \notag \\
  {\rm Lipschitz ~ Functions} & \Longrightarrow & {\rm e-convex} & \Longrightarrow & {\rm e-quasiconvex} \notag 
  \end{array} 
 \end{align*}
While the reverse implications does not hold in general as we show below.

\begin{example}
 \begin{itemize}
 \label{ex.3}
  \item[$(i)$] Let $f: \mathbb{R} \rightarrow \mathbb{R}$ be given by $f(x) = 
 \sqrt{\lvert x \rvert}$. Since $f$ is quasiconvex, it is $e$-quasiconvex for every error function $e$,  however, $f$ is not $e$-convex for any error bifunction $e$. Indeed, if $f$ is $e$-convex for some $e$, then $e_{f}$ is finite on ${\rm dom}\,f \times {\rm dom}\,f$ and it is $e_{f}$-convex. We show that for $x \neq 0$ we have $e_{f}(x, 0) = + \infty$. By definition,
\[
e_{f} (x, y) = \max\left\{  0, \sup_{t \in \, ]0,1[} \frac{f(tx + (1-t)y) - t f(x) - (1-t)
f(y)}{t(1-t)} \right\},
\]
thus,
\begin{align*}
 e_{f} (x, 0) & \geq \sup_{t \in ]0,1[} \frac{f(tx+(1-t)0)-tf(x) - (1-t)f(0)}{t(1-t)} \\
 & = \sup_{t \in ]0,1[}\frac{\sqrt{\lvert tx \rvert} - t\sqrt{\lvert x \rvert}}{t(1-t)} = \sup_{t \in ]0,1[}\frac{\sqrt{\lvert x \rvert}}{\sqrt{t}(1+\sqrt{t})}= + \infty.
\end{align*}

 \item[$(ii)$] Define $f: \mathbb{R} \rightarrow \mathbb{R}$ by $f(x) = - x^{2}$. 
 Clearly, $f$ is not quasiconvex, but it is $e$-quasiconvex with error bifunction 
 $e(x,y) = (x-y)^{2}$.
Indeed, we show that $f$ is $e$-convex. For a given $x,y\in%
\mathbb{R}
$ and $t\in \left[  0,1\right]  $ we have%
\begin{align*}
& f\left(  tx+\left(  1-t\right)  y\right)  -tf\left(  x\right)  -\left(
1-t\right)  f\left(  y\right)  \\
& =t\left(  x^{2}-\left(  tx+\left(  1-ty\right)  \right)  ^{2}\right)
+\left(  1-t\right)  \left(  y^{2}-\left(  tx+\left(  1-ty\right)  \right)
^{2}\right)  \\
& =t(1-t)(x-y)[(1+t)x+(1-t)y-tx+(t-2)y]\\
& =t(1-t)(x-y)^{2}.
\end{align*}
\end{itemize}
\end{example}

The next remark shows an instance of a function which is not $e$-quasiconvex for any error bifunction $e$.

\begin{remark}
Define $f: \mathbb{R} \rightarrow \mathbb{R}$ by 
$$ f (x) = \left\{
\begin{array}{ccl}
 1 & {\rm if} & x \in \, ]- \infty, - 3] \cup [1, + \infty[, \\
 \lvert x + 2 \rvert & {\rm if} & x \in \, ]-3,-1[, \\
 \sqrt{\lvert x \rvert} & {\rm if } & x \in [-1, 1].
\end{array}
\right.
$$
Note that $f$ is not $e$-quasiconvex for any error bifunction $e$. Indeed, su\-ppo\-se 
for the contrary that there exists an error bifunction $e: \mathbb{R} \times \mathbb{R}
\rightarrow \mathbb{R}$ for which $f$ is $e$-quasiconvex. By choosing $x=-2$, 
$y=0$ and $t = \frac{1}{n}$ we get $f(-2) = f(0) = 0$, and by definition of
 $e$-quasiconvexity, 
\[
f\left(  -\frac{2}{n}\right)  \leq \max\{  f(0), f(-2)\} + \frac{1}{n}\left(1 - \frac{1}{n}\right) e(0, -2).
\]
Consequently, for $n\geq2$ we derive
\begin{align*}
\sqrt{\frac{2}{n}}\leq \frac{1}{n} \left(  1-\frac{1}{n}\right) e (0, - 2) \,
 \Longleftrightarrow \, \sqrt{2} \leq \frac{1}{\sqrt{n}} (1 - \frac{1}{n}) 
 e (0, - 2).
\end{align*}
Letting $n \rightarrow \infty$, we obtain $\sqrt{2} \leq 0$, a contradiction. Therefore, 
$f$ is not $e$-quasiconvex for any error bifunction $e$.
\end{remark}

We recall the following result from \cite{Huang 24}, which will be used in the sequel.

\begin{theorem}{\rm (\cite[Theorem 3.6]{Huang 24})}\label{Th 3.6. Huang} 
 Let $e: \mathbb{R}^{n} \times \mathbb{R}^{n} \rightarrow \mathbb{R} \cup \left\{ + \infty \right\}$ be an error bifunction such that ${\rm dom}\,f \times {\rm dom}\,f \subseteq {\rm dom}\,e$ and that for every $x, y \in {\rm dom}\,f$, $e(\cdot, y)$ is convex and $e(x, x) = 0$. Suppose that ${\rm int}\left( {\rm dom}\,f \right) \neq \emptyset$ and $f$ is Fr\'{e}chet differentiable on the nonempty and open set $\Omega \subset {\rm dom}\,f$. If $f$ is $e$-quasiconvex, then for every $x, y \in \Omega$, the following implication holds
\begin{equation}
 f(y) \leq f(x) \, \implies \, \langle \nabla f(x), y - x \rangle \leq e(x, y).
\label{eq(conv=mon)}%
\end{equation}
Conversely, if relation \eqref{eq(conv=mon)} holds, then $f$ is a $2e$-quasiconvex function.
\end{theorem}

The following result outlines basic properties of $e$-quasiconvex functions. Its easy proof is omitted. We note that, in the following proposition, part (e) was given in (\cite[Proposition 3.12]{Huang 24} and we included it for the sake of completeness.

\begin{proposition}\label{properties}
\emph{(a)} Suppose that $f$ is $e$-quasiconvex and $e^{\prime}$ is an error bifunction such that ${\rm dom}\,f \times {\rm dom}\,f \subseteq {\rm dom}\,e^{\prime}$. If $e \leq e^{\prime}$, then $f$ is $e^{\prime}$-quasiconvex.

\emph{(b)} Let $f: \mathbb{R}^{n} \rightarrow \mathbb{R} \cup \left\{ + \infty \right\}$ be a function. Then $f$ is quasiconvex if and only if it is $e$-quasiconvex for every error bifunction $e$ such that ${\rm dom}\,f \times {\rm dom}\,f \subseteq {\rm dom}\,e$.

\emph{(c)} For each $e$-quasiconvex function $f$, ${\rm dom}\,f$ is convex.

\emph{(d)} Let $\{f_{i}\}_{i \in I}$ be a family of $e$-quasiconvex functions and 
 $g(x) := \sup_{i \in I} f_{i} (x)$ be proper. Then $g(x)$ is $e$-quasiconvex. 
    
\emph{(e)} Suppose $f$ is $e$-quasiconvex, $g$ is $e^{\prime}$-quasiconvex and for all $x, y \in {\rm dom}\,f \cap {\rm dom}\,g$ we have $(f(x)-f(y)(g(x)-g(y) )\geq 0$. Then $f+g$ is ($e + e^{\prime})$-quasiconvex. In particular, if $g$ is a quasiconvex function, then $f+g$ is $e$-quasiconvex.
\end{proposition}

In light of the above results, the theoretical framework relies on the choice of the map $e$. The subsequent results address this issue by discussing how to select the minimal error function and by providing an explicit formula for the map $e$. In particular, we analyze a lower critical bound that $e$ must satisfy. This analysis significantly strengthens the robustness of the proposed notion of $e$-quasiconvexity.

Given a function $f: \mathbb{R}^{n} \rightarrow \mathbb{R} \cup \left\{+ \infty \right\}$, we define the map $\widetilde{e}_{f}: \mathbb{R}^{n} \times \mathbb{R}^{n}\rightarrow \mathbb{R}_{+} \cup \left\{+ \infty \right\}$ as follows: For $x, y \in {\rm dom}\,f$
$$
\widetilde{e}_{f} (x, y) = \inf \left \{  a \in \mathbb{R}_{+}: \, \frac{f(tx + (1 - t) y)  -\max \left\{ f(x), f(y)  \right\}}{t(1 - t)} \leq a, ~ \forall ~ t \in \, ]0, 1[ \right\},
$$
and $\widetilde{e}_{f} (x, y) = + \infty$ if $(x, y) \notin {\rm dom}\,f \times {\rm dom}\,f$.

Clearly, $\widetilde{e}_{f}$ is a nonnegative symmetric error bifunction and $\widetilde{e}_{f} (x, x) = 0$ for all $x \in {\rm dom}\,f$. Hence, if $f$ is $e^{\prime}$-quasiconvex for some error bifunction $e^{\prime}$, then
\[
\widetilde{e}_{f}  =\inf \left\{e: \, f\text{ is }e\text{-quasiconvex}\right\},
\]
where the infimum is taken pointwise. In this case, $\widetilde{e}_{f}$ is finite on ${\rm dom}\,f \times {\rm dom}\,f$ and $f$ is
$\widetilde{e}_{f}$-quasiconvex. Note that $\widetilde{e}_{f}$ is the minimal $e$ such that $f$ is $e$-quasiconvex.

The next result provides an alternative formula for $\widetilde{e}_{f}$. 

\begin{proposition}
Suppose that $f$ is $e$-quasiconvex for some error bifunction $e$. Then for every $x, y \in {\rm dom}\,f$, we have
\begin{equation}\label{explisit}
 \widetilde{e}_{f} (x, y) = \max\left\{0, \sup_{t \in \, ]0,1[} \frac{f(tx + (1-t)y) - \max\{f(x), f(y)\}}{t(1-t)} \right\}.
\end{equation}
\end{proposition}

\begin{proof}
According to the definition of $e$-quasiconvexity, for every $x, y \in {\rm dom}\,f$, for all $t \in \, ]0,1[$  we have
$$f(tx+(1-t)y) -\max\{f(x), f(y)\} \leq t(1-t) e(x, y).$$
Multiplying both sides of the above inequality by $\frac{1}{t(1-t)}$ and taking the supremum of the left-hand side of the above inequality for $t \in \, ]0, 1[$,  we obtain 
$$\sup_{t \in ]0,1[} \frac{f(tx+ (1-t)y)- \max\{f(x), f(y)\}}{t(1-t)} \leq e(x, y).$$
From the above inequality and non-negativity of $e$ we get
$$\max\left\{0,\sup_{t \in \, ]0,1[} \frac{f(tx+(1-t)y) - \max\{f(x), f(y)\}}{t(1-t)}\right\} \leq e(x, y).$$
Now from the minimality $\widetilde{e}_{f}$, we conclude (\ref{explisit}).
\end{proof}

The following result is useful for checking  $e$-quasiconvexity.

\begin{proposition}\label{fornot:eqcx}
Suppose that $f: \mathbb{R}^{n} \rightarrow \mathbb{R} \cup \left \{+\infty \right\}$ is a function for which the equation $f(x) = 0$ has at least two distinct roots, denoted as $x_{1}$ and $x_{2}$. Then the following assertions hold:
\begin{itemize}
 \item[$(a)$] If
 \begin{equation}\label{ness-eq}
  \underset{t \rightarrow 0^{+}}{\lim \sup} \frac{f(tx_{1} + (1-t)x_{2})}{t} = + \infty,
 \end{equation}
 then there is no error bifunction $e$ for which $f$ is $e$-quasiconvex.

 \item[$(b)$] If $f$ is $e$-quasiconvex for some error bifunction $e$, then
 \[
 \underset{t \rightarrow 0^{+}}{\lim \sup} \frac{f(t x_{1} + (1-t) x_{2})}{t} \leq e(x_{1}, x_{2}).
 \]
 \end{itemize}
\end{proposition}

\begin{proof}
$(a)$: For the sake of contradiction, assume that $f$ is $e$-quasiconvex for some error bifunction $e$. Then
\begin{align}
 & f(t x_{1} + (1-t) x_{2}) \leq \max\left\{ f(x_{1}), f(x_{2})\right\} + t (1- t)  e(x_{1}, x_{2}) = t(1-t) e(x_{1}, x_{2}) \notag \\
 & \hspace{2.6cm} \Longrightarrow \, \frac{f(t x_{1} + (1-t) x_{2})}{t} \leq (1-t) e(x_{1}, x_{2}). \label{lim-eq}
\end{align}
By taking the limit superior from the above inequality when $t \rightarrow 0^{+}$, and using \eqref{ness-eq} we deduce that $e(x_{1}, x_{2}) \geq + \infty$, a contradiction.

$(b)$: Take the limit superior to \eqref{lim-eq} and conclude the desired inequality.
\end{proof}

\begin{example}
Define $f: \mathbb{R} \rightarrow \mathbb{R}$ by $f(x) = \sin(\sqrt{\lvert x \rvert})$. Clearly, $f$ is not quasiconvex. Let us check if $f$ is $e$-quasiconvex or not. Since $f(\pi^{2}) = f(0) = 0$ and the $\underset{t \rightarrow 0^{+}}{\lim \sup} \frac{f(t \pi^{2})}{t} = + \infty$, it follows from Proposition \ref{fornot:eqcx} that $f$ is not $e$-quasiconvex for any error bifunction $e$.
\end{example}

In the following result, we characterize $e$-convex functions in terms of the sum 
of $e$-quasiconvex with linear functions, a result which extends the 
cha\-rac\-te\-ri\-za\-tion provided in \cite{ACJ 94} for convex and quasiconvex 
functions.

\begin{proposition}\label{char:eqcx}
 A function $f: \mathbb{R}^{n} \rightarrow \mathbb{R} \cup \left\{ +\infty \right\}$ if $e$-convex if and only if for all $x^{\ast} \in \mathbb{R}^{n}$, the function $f + x^{\ast}$ is $e$-quasiconvex.
\end{proposition}

\begin{proof}
$(\Rightarrow)$: Since $x^{\ast}$ is convex, it fo\-llows from \cite[Proposition 
5-(d)]{Ali 203e-conv} that $f+x^{\ast}$ is $e$-convex, so $e$-quasiconvex.

$(\Leftarrow)$: Suppose that $f + x^{\ast}$ is $e$-quasiconvex for every $x^{\ast} \in \mathbb{R}^{n}$. Hence, for every $x, y \in \mathbb{R}^{n}$, one can choose $x^{\ast} \in \mathbb{R}^{n}$ such that
\begin{equation}\label{lin-f}
 \langle x^{\ast}, y-x \rangle = f(x) - f(y).
\end{equation}
Therefore, $(f + x^{\ast})(x) = (f + x^{\ast}) (y)$. Then for all $t \in [0, 1]$, we get
\begin{align*}
 (f+x^{\ast}) (x) + t(1-t) e(x, y) & \geq (f + x^{\ast}) (ty + (1-t) x) \\
 &  = f (ty + (1-t) x) + \langle x^{\ast}, x \rangle + t \langle x^{\ast}, y-x \rangle \\
 & = f(ty + (1-t)x) + \langle x^{\ast}, x \rangle + t (f(x)-f(y)).
\end{align*}
Hence,
$$f(ty +(1-t)x) \leq t f(y) + (1-t) f(x) + t(1-t) e(x, y), ~ \forall ~ t \in [0, 1],$$
i.e., $f$ is $e$-convex.
\end{proof}

Recall that $x_{0} \in {\rm dom}\,f$ is a global minimizer of $f$ if $f(x_{0}) =
\inf_{\mathbb{R}^{n}} f$, and a (strict) local minimizer of $f$ if there exists 
$\varepsilon > 0$ such that
\begin{align*}
 & ( f(  x_{0}) < f(y), ~ \forall ~ y \in B(x_{0}, \varepsilon)  \text{ and } y 
\neq x_{0} ), \\
 & ~~~ f(x_{0}) \leq f(y), ~ \forall ~ y \in B (x_{0}, \varepsilon) .
\end{align*}

As is well known (see, for instance, \cite[Theorem 1.32]{Ansari-book}), every strict local minimum of a quasiconvex function is also a strict global minimum. While $e$-quasiconvexity offers a useful unification of $e$-convex and quasiconvex functions, it also introduces certain technical difficulties. The following example shows that, in general, a strict local minimum of an $e$-quasiconvex function need not be a strict global minimum.

\begin{example}
Define the function $f: \mathbb{R} \rightarrow \mathbb{R}$ by
\[
 f(x)  =\left \{
\begin{array}{ccl}
 x-4 & {\rm if} &  x>2, \\
 -3x+4 ~~ & {\rm if} & 1 \leq x\leq2, \\
 \left \vert x\right \vert  & {\rm if} & x<1.
\end{array}
\right.
\]

This function is not quasiconvex, but it is $e$-quasiconvex because it is $e$-convex. 
Indeed, note that 
$$f(x) = 2 \lvert \lvert x-1 \rvert - 1 \rvert -\lvert x \rvert.$$ 
According to \cite[Example 2.10]{AZ 2023}, $g(x) = 2 \lvert \lvert x-1 \rvert -1 \rvert$ is $e$-convex with $e(x, y) = 4 \lvert x-y \rvert$. Furthermore, as stated in \cite[Example 4]{Ali 2021}, the function $h(x) = - \lvert x \rvert$ is $e$-convex with $e(x, y) = 2 \lvert x-y \rvert $. Therefore, $f=g+h$ is $e$-convex with $e(x, y) = 6 \lvert x-y \rvert$.
Consequently, $f$ is $e$-quasiconvex with $e (x, y) = 6 \lvert x-y \rvert $. Finally, note that ${\rm argmin}_{\mathbb{R}}\,f = \left \{0, 2\right\}  $, where $0$ is a strict local minimizer, but not a global minimizer of $f$.
\end{example}

%
%

Further theoretical properties, examples, and applications in optimization problems of $e$-quasiconvex functions may be found in \cite{Huang 24}.

\section{Nonemptiness and Compactness of the Solution Set}\label{sec:04}

As noted in the preliminaries, asymptotic/recession tools are extremely useful for 
providing existence of solutions (see \cite{ABM,AT,BBGT}). In this section, we 
develop these results for $e$-quasiconvex functions in the constraint and unconstraint cases. In what follows, we write {\it usc} for upper semicontinuous.

\begin{theorem}\label{exist:result1}
Let $f: \mathbb{R}^{n} \rightarrow \mathbb{R} \cup \left \{  +\infty \right \}$ be a proper lsc and $e$-quasiconvex function such that the error bifunction $e$ is usc and satisfies that 
\begin{equation}\label{assump:e}
 e(\lambda x, \lambda y) = \lambda e(x, y), ~ \forall ~ \lambda \geq 0. 
\end{equation}
If
\begin{equation}\label{suff:cond}
 f^{\infty}_{q} (u) - e(u, 0) > 0, ~ \forall ~ u \neq 0,
\end{equation}
then ${\rm argmin}_{\mathbb{R}^{n}}\,f \neq \emptyset$ and compact. 
\end{theorem}

\begin{proof}
 (Existence): By the definition of infimum, there exists a minimizing sequence $\{x_{k}\} \subseteq 
 {\rm dom}\,f$ such that $f(x_{k}) \rightarrow \inf_{\mathbb{R}^{n}}\,f$. 

 We claim that $\{x_{k}\}_{k}$ is bounded. Indeed, if not, then we may assume that $\lVert x_{k} \rVert \rightarrow + \infty$, thus (take a subsequence if needed) $\frac{x_{k}}{\lVert x_{k} \rVert} \rightarrow u \in ({\rm dom}\,f)^{\infty}$, with $\lVert u \rVert = 1$. Take any $y \in {\rm dom}\,f$ and any $t>0$. Since $\lVert x_{k} \rVert \rightarrow + \infty$, there exists $k_{0} \in \mathbb{N}$ such that
 $$0 < \frac{t}{\lVert x_{k} \rVert} < 1, ~ \forall ~ k \geq k_{0}.$$
 Since $\{x_{k}\}_{k}$ is a minimizing sequence, there exists $k_{1} \in \mathbb{N}$ such that 
 $$f(x_{k}) \leq f(y), ~ \forall ~ k \geq k_{1}.$$ 
 Hence, for every $k \geq k_{2} := \max\{k_{0}, k_{1}\}$, we have
 \begin{align*}
  f\left( \left( 1 - \frac{t}{\lVert x_{k} \rVert} \right) y + \frac{t}{\lVert x_{k} \rVert} x_{k} \right) & \leq \max\{f(x_{k}), f(y)\} + \frac{t}{\lVert x_{k} \rVert} \left( 1 - \frac{t}{\lVert x_{k} \rVert} \right) e(x_{k}, y) \\
   & = f(y) + t \left( 1 - \frac{t}{\lVert x_{k} \rVert} \right) e\left(\frac{x_{k}}{\lVert x_{k} \rVert}, \frac{y}{\lVert x_{k} \rVert} \right),
 \end{align*}
where the equality holds in virtue of assumption \eqref{assump:e}.
Since $f$ is lsc and $( 1 - \frac{t}{\lVert x_{k} \rVert}) y + \frac{t}{\lVert x_{k} \rVert} x_{k} \rightarrow y + tu$ as $k \rightarrow + \infty$, taking the $\liminf_{k}$ in the previous relation and using the upper semicontinuity of $e$, we deduce 
 \begin{align*}
  & f(y+tu) \leq f(y) + t e(u, 0), ~ \forall ~ t>0, ~ \forall ~ y \in {\rm dom}\,f \\
  & \Longleftrightarrow \, \sup_{y \in {\rm dom}\,f} \sup_{t>0} \frac{f(y+tu)-f(y)}{t} \leq e(u, 0) \\
  & \Longleftrightarrow \, f^{\infty}_{q} (u) \leq e(u, 0),
 \end{align*} 
 a contradiction (because $u \neq 0)$. Therefore, the sequence $\{x_{k}\}_{k}$ is bounded. Hence, there exists a subsequence $\{x^{1}_{k}\}_{k}$ 
such that $x^{1}_{k}\rightarrow \overline{x} \in {\rm dom}\,f$, and since $f$ is lsc, 
 $$f(\overline{x}) \leq \liminf_{k \rightarrow + \infty} f(x^{1}_{k}) = \inf_{\mathbb{R}^{n}}\,f,$$
 Therefore, $\overline{x} \in {\rm argmin}_{\mathbb{R}^{n}}\,f$.

(Compactness): Suppose for the sake of contradiction that ${\rm argmin}_{\mathbb{R}^{n}}\,f$ is unbounded. Then there 
exists a sequence $\{x_{k}\}_{k} \subseteq {\rm argmin}_{\mathbb{R}^{n}}\,f$ with $\lVert x_{k} \rVert \rightarrow + \infty$ 
such that $\frac{x_k}{\lVert x_{k} \rVert} \rightarrow u \in ({\rm dom}\,f)^{\infty}$ for some $u \neq 0$. 

Take $y \in {\rm dom}\,f$ and $t>0$. Since $f$ is $e$-quasiconvex and $\{x_{k}\}_{k} \subseteq {\rm argmin}_{\mathbb{R}^{n}}\,f$, we have
$$f\left( \left(1-\frac{t}{\lVert x_{k} \rVert}\right)y + \frac{t}{\lVert x_{k} \rVert} x_{k} \right) \leq f(y) + \frac{t}{\lVert x_{k} \rVert} \left( 1 - \frac{t}{\lVert x_{k} \rVert} \right) e(x_{k}, y).$$
Since $f$ is lsc and $e$ is usc, using \eqref{assump:e} in the above inequality we obtain
$$f(y+tu) \leq f(y) + t e(u, 0) \, \Longleftrightarrow \, f^{\infty}_{q} (u) \leq e(u, 0),$$
a contradiction to \eqref{suff:cond}. Therefore, ${\rm argmin}_{\mathbb{R}^{n}}\,f$ is bounded. 

Finally, since $f$ is lsc, we obtain that ${\rm argmin}_{\mathbb{R}^{n}}\,f$ is closed, thus compact, which completes the proof.
\end{proof}

The following example is an instance in which Theorem \ref{exist:result1} applies 
and no other result (as far as we know) does.

\begin{example}\label{exam:exist}
 Let us consider the function $f: \mathbb{R} \rightarrow \mathbb{R}$ given by 
 $$f(x) = \min \left \{ g(x), 6 \right \},$$ where $g(x) = \max \left \{  \lvert \lvert x \rvert - 1 \rvert, x^{2} - 3 \right \}$ (see Figure \ref{fig:exist}). 
\begin{figure}[htbp]
\centering
\includegraphics[scale=0.55]{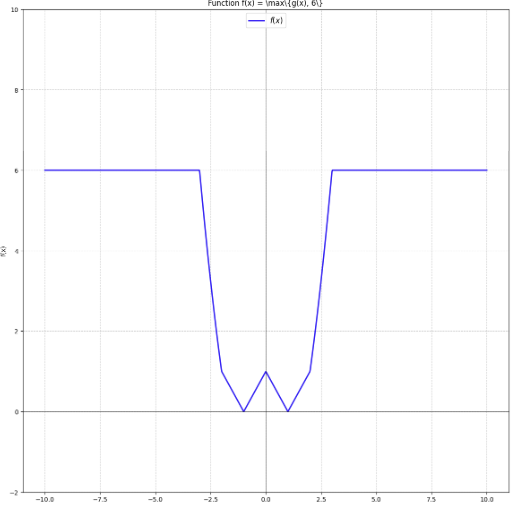}
\caption{Function $f$ in Example \ref{exam:exist}.}
\label{fig:exist}
\end{figure}

Clearly, ${\rm argmin}_{\mathbb{R}}\,f=\{-1, 1\}$ is nonempty and compact, but it is not convex. Furthermore, $f$ is continuous, noncoercive and it is not quasiconvex, i.e., the existence results from \cite{Aus2,AT,ABM,FFB-Vera,HLM,HLL,Iusem-Lara-5} can not be applied.

We claim that $f$ is $e$-quasiconvex with $e(x, y) = 2 \lvert x - y \rvert$. Indeed, we know by \cite[Example 5]{Ali 2021} that $g_{1}(x) = \lvert \lvert x \rvert - 1 \rvert$ is $e$-convex with $e(x, y) = 2 \lvert x - y \rvert$ while $g_{2}(x) = x^{2} - 3$ is convex, thus $e$-convex for every error bifunction $e$, in par\-ti\-cu\-lar, for $e(x, y) = 2 \lvert x - y \rvert$. Then, by \cite[Remark 7]{Ali 2018}, $g$ is also $e$-convex with $e(x, y) = 2 \lvert x - y \rvert$. 

Now, we show that $f$ is $e$-convex with $e(x, y) = 2 \lvert x - y \rvert$. According to the definition of $f$, we have for every $x \in \mathbb{R}$ that $f(x) \leq 6$. Then, we consider two cases:
\begin{itemize}
 \item[$(i)$] Suppose that $\lvert x \rvert$ or $\lvert y \rvert$ is greater than $3$. Then, $\max \left\{f(x), f(y) \right\}  =6$. Therefore,
 $$f(tx + (1 - t) y) \leq 6 \leq \max \left \{  f(x), f(y)  \right \} + 2 \lvert x - y \rvert.$$

\item[$(ii)$] Suppose that $\lvert x \rvert \leq 3$ and $\lvert y \rvert \leq 3$. In this case, $f(x) = g(x)$, $f(y) = g(y)$. In addition, for every $t \in \, ]0,1[$ we have $\lvert tx + (1 - t) y \rvert \leq 3$, and thus $f (tx + (1-t) y) = g((tx + (1 - t)y))$. Hence
\begin{align*}
 f(tx + (1-t ) y) & = g (tx + (1 - t) y)  \\
& \leq t g (x) + (1-t) g(y) + 2 \lvert x - y \rvert \\
& =t f(x) + (1 - t)  f(y) +2 \lvert x -  y \rvert.
\end{align*}
\end{itemize}
Therefore, $f$ is $e$-quasiconvex with $e(x, y)= 2\lvert x - y \rvert$.

Now, let us compute $f^{\infty}_{q}(u)$. Observe that for $u=1$, we have
\begin{align*}
 f^{\infty}_{q} (1) & = \sup_{x \in \mathbb{R}} \sup_{t>0} \frac{f(x+t) - f(x)}{t} 
 \geq \lim_{t \downarrow 0} \frac{f(2.9 + t) - f(2.9)}{t}\\
 & = \lim_{t \downarrow 0} \frac{( (2.9)^{2} + 2 \cdot 2.9 \cdot t + t^{2} - 3)  
 - ((2.9)^{2}-3)}{t} \\
 & = \lim_{t \downarrow 0} (5.8 + t) \\
 & = 5.8.
\end{align*}
Analogously, for $u=-1$, we have $f^{\infty}_{q} (-1) \geq 5.8$.
 
Finally, since $e(u, 0) = 2 \lvert u \rvert = 2$, $e$ is usc and satisfies assumption \eqref{assump:e}, 
$$f^{\infty}_{q} (u) - e(u, 0) \geq 5.8 - 2 = 3.8 > 0, ~ \forall ~ u \neq 0,$$
i.e., it follows from Theorem \ref{exist:result1} that ${\rm argmin}_{\mathbb{R}}\,f$ 
is nonempty and compact.
\end{example}

\begin{remark}
 The necessary implication in Theorem \ref{exist:result1} is still open, but when $e=0$ (that is, $f$ is quasiconvex), Theorem \ref{exist:result1} reduces to the sufficient condition provided in \cite[Theorem 4.7]{FFB-Vera}, which is also necessary in this case. 
\end{remark}

%

Theorem \ref{exist:result1} can be adapted to the constraint case, as we note next. To that end, let us define the set:
\begin{equation}\label{set:Kq}
 K_{q} (f, e) := \{u \in \mathbb{R}^{n}: ~ f^{\infty}_{q} (u) \leq e(u, 0)\}.
\end{equation}

\begin{proposition}\label{suff:constraint} 
 Suppose $\Omega \subseteq \mathbb{R}^{n}$ is a closed and convex set, $e$ is an error bifunction such that is usc and satisfies assumption \eqref{assump:e}, and $f$ is a lsc and $e$-quasiconvex function with $\Omega \subseteq {\rm dom}\,f$. If
 \begin{equation}\label{useful:suff}
  \Omega^{\infty} \cap K_{q} (f, e) = \{0\},
 \end{equation}
 then ${\rm argmin}_{\Omega}\,f$ is nonempty and compact.
\end{proposition} 

\begin{proof} 
 By Theorem \ref{exist:result1}, it suffices to show that condition \eqref{suff:cond} is equivalent to 
\begin{equation}\label{eq:const} 
 (f + \iota_{\Omega})^{\infty}_{q} (u) - e(u, 0) > 0, ~ \forall ~ u \in \Omega^{\infty} \backslash \{0\},
\end{equation} 
where $\iota_\Omega$ is the indicator function of $\Omega$. Indeed, by definition of $q$-asymptotic function and the convexity of $\Omega$, for all $d \in \mathbb{R}^n$, we have
\begin{align*}
 (f + \iota_{\Omega} )_q^\infty (d) & = \sup_{x \in {\rm dom}\,(f + \iota_{\Omega})} \sup_{t>0} \frac{(f + \iota_{\Omega}) (x + td) - (f + \iota_{\Omega}) (x)}{t} \\
 & = \sup_{x \in \Omega} \sup_{t>0} \frac{f(x+td) + \iota_{\Omega} (x+td) - f(x)}{t} \\
 & =
 \begin{cases}
  f_q^\infty (d) &\text{if} \ \ d \in \Omega, \\
  +\infty &\text{otherwise},
 \end{cases}
 \\
 & = f_q^\infty (d) + \iota_{\Omega^\infty}(d).
\end{align*}
Hence, $\Omega^{\infty} \cap K_{q} (f, e) = \{0\}$ if and only if $K_q
(f + \iota_{\Omega}, e) = \{0\}$.
\end{proof}

\section{Sufficient KKT Optimality Conditions}\label{sec:05}

Generalized convexity is important for pro\-vi\-ding sufficiency of the KKT optimality
conditions (see \cite{AE,Manga}) in mathematical programming problems. In this 
section, we show that KKT conditions could be sufficient for optimality under an 
$e$-quasiconvexity assumption on the constraints functions.

Let $K \subseteq \mathbb{R}^{n}$ be an open convex set, $f: K \rightarrow 
\mathbb{R}$ be a real-valued function, $g_{i}: K \rightarrow \mathbb{R}$ be 
real-valued functions for every $i \in I := \{1,\ldots,m\}$ and $h_{j}: K 
\rightarrow \mathbb{R}$ be real-valued functions for every $i \in J := \{1,\ldots,k\}$. 

We are interested in the nonconvex mathematical programming problem:
\begin{equation}\label{extended:problem}
 \min_{x \in C^{\prime}} f(x), \tag{P}
\end{equation}
where $C^{\prime} := \{x \in K: ~ g_{i} (x) \leq 0, ~ \forall ~ i \in I, h_{j} 
(x) = 0, ~ \forall ~ j \in J\}$.

As usual, the set of active index at a point $\overline{x} \in C^{\prime}$ is denoted by
\begin{equation}\label{active: set}
 I(\overline{x}) := \{i \in I: ~ g_{i} (\overline{x}) =0 \}.
\end{equation}

Let  $p, q \in \mathbb{N}$ and $S  \subseteq \mathbb{R}^{p}$ be a nonempty set. Then, for a given function $f: S \rightarrow \mathbb{R}^{q}$, $\overline{x} \in S$ and $w \in \mathbb{R}^{p}$, if the limit
\begin{equation}\label{dirdev:vector}
f^{\prime} (\overline{x})(w) := \lim_{t \rightarrow 0^{+}} \frac{f(\overline{x} + tw)  -f(\overline{x})}{t}
\end{equation}
exists, then $f^{\prime} (\overline{x}) (w)$ is the directional derivative of $f$ at $\overline{x}$ in the direction $w$. 

If this limit \eqref{dirdev:vector} exists and finite for all $w \in \mathbb{R} ^{p}$, then $f$ is called directionally differentiable at $\overline{x}$. Furthermore, we also introduce the following notion:

\begin{definition}\label{strong:functional}
 Let $S \subseteq \mathbb{R}^{n}$ be nonempty, $\widehat{C} \subseteq \mathbb{R}^{m}$ be nonempty, $e := (e_{1}, \ldots, e_{m})$ with $e_{i}: \mathbb{R}^{n} \times \mathbb{R}^{n} \rightarrow \mathbb{R}$ an error bifunction for every $i=1, \ldots, m$, and $\varphi: S \rightarrow \mathbb{R}^{m}$ be a mapping with directional derivative at some $\overline{x} \in S$ for every direction $x - \overline{x}$ with $x \in S$. Then the vector mapping $\varphi$ is called $\widehat{C}$-$e$-qua\-si\-con\-vex at $\overline{x}$, if for all $x \in S$, we have
 \begin{equation}\label{eq:strongfunctional}
  \varphi(x) - \varphi(\overline{x}) \in \widehat{C} ~ \Longrightarrow ~
  \varphi^{\prime} (\overline{x}) (x - \overline{x}) - e(x, \overline{x}) \in
  \widehat{C}.
 \end{equation}
\end{definition}

Our next result extends \cite[Theorem 5.14]{J-2020}.

\begin{theorem}\label{main:theo2}
 Let $K$, $f$, $g_{i}$, $h_{j}$ and $C^{\prime}$ be defined as above and $e := (e_{1}, \ldots, e_{m})$ with $e_{i}: \mathbb{R}^{n} \times \mathbb{R}^{n} \rightarrow \mathbb{R}$ an error bifunction for every $i=1, \ldots, m$. Suppose that $f$, $g := (g_{1}, \ldots, g_{m})$ and $h := (h_{1}, \ldots, h_{k})$ have a directional derivative at $\overline{x} \in C^{\prime}$ in every direction in every direction  $x$ and  all the functions are $\text{Gâteaux differentiable}$.
 Moreover, assume that there exists $(\overline{u}, \overline{v}) \in \mathbb{R}^{m}_{+} \times \mathbb{R}^{k}$, $(\overline{u}, \overline{v}) \neq (0, 0)$, with \begin{align}
  & \left\langle  f' (\overline{x}) + \sum^{m}_{i=1} \overline{u}_{i} g'_{i} (\overline{x}) + \sum^{k}_{j=1} \overline{v}_{j}  h'_{j} (\overline{x}), x - \overline{x} \right\rangle \geq \sum^{m}_{i=1} \overline{u}_{i} e_{i} (x, \overline{x}), ~ \forall \, x \in K
  \label{for:contra} \\
  & \hspace{4.9cm} \sum^{m}_{i=1} \overline{u}_{i} g_{i} (\overline{x}) = 0. \label{cond}
 \end{align}
 Then $\overline{x}$ is a global minimum for $f$ on
 \begin{equation}
 \widetilde{C} : =\{x \in K: \, g(x) \in - \mathbb{R}^{m}_{+} + \mathcal{L}(g(\overline{x})), \, h(x) = 0_{k} \},
\end{equation} 
 if and only if the map $(f, g, h): K \rightarrow \mathbb{R} \times \mathbb{R}^{m} 
 \times \mathbb{R}^{k}$, is $\widehat{C}$-$\widehat{e}$-quasiconvex at 
 $\overline{x}$ (where $\widehat{e} := (0, e, 0_{k})$) with 
 \begin{equation}\label{C:hat}
  \widehat{C} := (\mathbb{R}_{-} \backslash \{0\}) \times (- \mathbb{R}^{m}_{+} +\mathcal{L}(g(\overline{x})) ) \times \{0_{k}\}.
 \end{equation}
\end{theorem}

\begin{proof}
 First, we show that:
 \begin{equation}\label{claim:01}
 \left( f^{\prime} (\overline{x}) (x - \overline{x}), g^{\prime} (\overline{x})
 (x - \overline{x}) - e (x, \overline{x}), h^{\prime} (\overline{x}) (x - 
 \overline{x}) \right) \not\in \widehat{C}, ~ \forall ~ x \in K.
 \end{equation}
 
 Indeed, suppose on the contrary that there exists $x_{0} \in K$ such that
 $$\left( f^{\prime} (\overline{x}) (x_{0} - \overline{x}), g^{\prime} (\overline{x}) (x_{0} - \overline{x}) - e (x_{0}, \overline{x}), h^{\prime} (\overline{x}) (x_{0} - \overline{x}) \right) \in \widehat{C},$$
 that is,
 \begin{eqnarray}
  f^{\prime} (\overline{x}) (x_{0} - \overline{x}) & < & 0, \notag \\
  g^{\prime} (\overline{x}) (x_{0} - \overline{x}) - e (x_{0}, \overline{x}) 
  & \in & - \mathbb{R}^{m}_{+} + \mathcal{L}(g(\overline{x})) \label{MF:01} \\
  h^{\prime} (\overline{x}) (x_{0} - \overline{x}) & = & 0_{k}. \notag
 \end{eqnarray} 
 Then we have for some $\alpha, \beta \geq 0$ that 
 \begin{align*}
  f^{\prime} (\overline{x}) (x_{0} - \overline{x})  & + \sum^{m}_{i=1} 
  \left( \overline{u}_{i} g^{\prime}_{i} (\overline{x}) (x_{0} - \overline{x}) 
  - \overline{u}_{i} e_{i} (x_{0}, \overline{x}) \right) + \sum^{k}_{i=1}
  \overline{v}_{i} h^{\prime}_{i} (\overline{x}) (x_{0} - \overline{x}) \notag \\
  & < 0 + \sum^{m}_{i=1} \overline{u}_{i} \left(  g^{\prime}_{i} (\overline{x}) (x_{0} - \overline{x}) - e_{i} (x_{0}, \overline{x}) \right) + 0 \notag \\
  & \leq \sum^{m}_{i=1} \overline{u}_{i} \left( - \lvert y_{i} \rvert + \alpha g_{i}(\overline{x}) \right)
  ~~  (by ~ \eqref{MF:01} ~ with ~ \lvert y_{i} \rvert \in \mathbb{R}_{+})  \\
  & = - \sum^{m}_{i=1} \overline{u}_{i} \lvert y_{i} \rvert \leq 0, \notag
 \end{align*} 
 
 which contradicts \eqref{for:contra}. Therefore, relation \eqref{claim:01} holds.

 $(\Leftarrow)$: If the map $(f, g, h)$ is $\widehat{C}$-$\widehat{e}$-quasiconvex at $\overline{x}$ (with $\widehat{e} = (0, e, 0_{k})$), then it follows from \eqref{claim:01} that $(f(x) - f(\overline{x}), g(x) - g(\overline{x}), h(x) - h(\overline{x})) \not\in \widehat{C}$, that is, there is no $x \in K$ such that
\begin{eqnarray}
 f(x) & < & f(\overline{x}) \notag \\
 g(x) & \in & \{g(\overline{x})\} - \mathbb{R}^{m}_{+} + \mathcal{L}(g(\overline{x})) \subseteq - \mathbb{R}^{m}_{+} + \mathcal{L}(g(\overline{x})) \notag \\
 h(x) & = & 0_{k}. \notag
\end{eqnarray}

Since $\{g(\overline{x})\} \in - \mathbb{R}^{m}_{+} \subseteq -
\mathbb{R}^{m}_{+} + \mathcal{L}(g(\overline{x}))$, $\overline{x} \in \widetilde{C}$, hence $\overline{x}$ is 
a global minimum for $f$ on $\widetilde{C}$.

$(\Rightarrow)$ Suppose that $\overline{x}$ is a global minimum for $f$ on $\widetilde{C}$, then there is no $x \in \widetilde{C}$ such that
\begin{eqnarray}
 f(x) & < & f(\overline{x}) \notag \\
 g(x) & \in & - \mathbb{R}^{m}_{+} + \mathcal{L}(g(\overline{x})) \, = \{g(\overline{x})\} - \mathbb{R}^{m}_{+} + \mathcal{L}(g(\overline{x})) \notag \\
 h(x) & = & 0_{k}, \notag
\end{eqnarray}
that is,
$$(f(x) - f(\overline{x}), g(x) - g(\overline{x}), h(x) - h(\overline{x})) \not\in 
\widehat{C}, ~ \forall ~ x \in K.$$

Therefore, by this and relation \eqref{claim:01}, we conclude that the map 
$(f, g, h)$ is $\widehat{C}$-$\widehat{e}$-quasiconvex at $\overline{x}$ with 
$e = (0, e, 0_{k})$, which completes the proof.
\end{proof}

A sufficient condition for global minimizers in problem
 \eqref{extended:problem} is given below.

\begin{corollary}\label{jahn}
 Let $K$, $f$, $g_{i}$, $h_{j}$, $C^{\prime}$ be defined as above and $e := (e_{1}, \ldots, e_{m})$ with $e_{i}: \mathbb{R}^{n} \times \mathbb{R}^{n} \rightarrow \mathbb{R}$ an error bifunction for every $i=1, \ldots, m$. Suppose that $f$, $g := (g_{1}, \ldots, g_{m})$ and $h := (h_{1}, \ldots, h_{k})$ have a directional derivative at $\overline{x} \in C^{\prime}$ in every direction $x - \overline{x}$ with arbitrary $x \in K$. Moreover, assume that there exists $(\overline{u}, \overline{v}) \in \mathbb{R}^{m}_{+} \times \mathbb{R}^{k}$, $(\overline{u}, \overline{v}) \neq (0, 0)$ with 
 \begin{align*}
  & \left\langle  f'(\overline{x}) + \sum^{m}_{i=1} \overline{u}_{i} g'_{i} (\overline{x}) + \sum^{k}_{j=1} \overline{v}_{j} h'_{j} (\overline{x}), x - \overline{x} \right\rangle \geq \sum^{m}_{i=1} 
  \overline{u}_{i} e_{i} (x, \overline{x}), ~ \forall ~ x \in K, \\ 
  & \hspace{4.9cm} \sum^{m}_{i=1} \overline{u}_{i} g_{i} (\overline{x}) = 0, 
 \end{align*}
 and that the map $(f, g, h): K \rightarrow \mathbb{R} \times \mathbb{R}^{m} \times \mathbb{R}^{k}$, is $\widehat{C}$-$\widehat{e}$-quasiconvex at $\overline{x}$ with $\widehat{e} := (0, e, 0_{k})$ and $\widehat{C}$ as in \eqref{C:hat},
 then $\overline{x}$ is a minimal point for $f$ on $C^{\prime}$.
\end{corollary}


In particular, we have the following application  when $h \equiv 0$ in 
problem \eqref{extended:problem}. Hence, in this case, the feasible set becomes 
$C := \{x \in K: ~ g_{i} (x) \leq 0, ~ \forall ~ i \in I\}$.

The following result extends the Theorem in  \cite[pages 151-152]{Manga} from quasiconvex to $e$-quasiconvex constraints.

\begin{theorem}\label{theo:suff}
 Let $K$, $f$, $g_{i}$ and $C$ be defined as above and $e := (e_{1}, \ldots, e_{m})$ with $e_{i}: \mathbb{R}^{n} \times \mathbb{R}^{n} \rightarrow \mathbb{R}$ an error bifunction for every $i=1, \ldots, m$. Suppose in addition that $C$ is convex, $f$ is pseudoconvex, $g := (g_{1}, \ldots, g_{m})$ is $e$-quasiconvex and that $f$ and $g$ are di\-ffe\-ren\-tia\-ble at $\overline{x} \in K$. If there exists $\overline{u} \in \mathbb{R}^{m}$ such that $(\overline{x}, \overline{u}) \in \mathbb{R}^{n} \times \mathbb{R}^{m}$ with
  \begin{eqnarray}
  \left\langle \nabla f(\overline{x}) + \sum^{m}_{i=1} \overline{u}_{i} 
  \nabla g_{i} (\overline{x}), \, x - \overline{x} \right\rangle \nonumber
  & \geq & \sum^{m}_{i=1} \overline{u}_{i} e_{i} (x, \overline{x}), ~ \forall ~ x \in C, \nonumber \\
  \sum^{m}_{i=1} \overline{u}_{i} \, g_{i} (\overline{x}) & = & 0, 
  \nonumber \\ 
  g_{i} (\overline{x}) & \leq & 0, ~ \forall ~ i \in I, \nonumber \\
  \overline{u} & \geq & 0. \label{KKT:point}
\end{eqnarray} 
 Then $\overline{x}$ is a global minimum for problem \eqref{extended:problem} (with 
 $h \equiv 0)$.
\end{theorem}

\begin{proof}
 Let $I(\overline{x})$ be the active set of $\overline{x} \in K$ and $K(\overline{x}) := \{i \in I: ~ g_{i} (\overline{x}) < 0 \}$. Then $I = I(\overline{x}) \cup K(\overline{x})$. Furthermore, since $\overline{u}_{i} \geq 0$ and
 $g_{i} (\overline{x}) \leq 0$ for all $i \in I$, and $\sum^{m}_{i=1} \overline{u}_{i} g_{i} (\overline{x}) = 0$, then
 $$\overline{u}_{i} g_{i} (\overline{x}) = 0, ~ \forall ~ i \in I ~ \Longrightarrow ~ \overline{u}_{j} = 0, ~ \forall ~ j \in K(\overline{x}).$$

 Since $g_{i} (x) \leq 0 = g_{i} (\overline{x})$ for all $i \in I(\overline{x})$ 
 and every $g_{i}$ is $e_{i}$-quasiconvex, it follows from Theorem  \ref{Th 3.6. Huang} that
 \begin{align}
  & \hspace{1.1cm} \langle \nabla g_{i} (\overline{x}), x - \overline{x} \rangle 
  \leq e_{i} (x, \overline{x}), ~ \forall ~ i \in I(\overline{x}), ~ \forall ~ x \in 
  C \notag \\
  & \Longrightarrow \, \overline{u}_{i} \langle \nabla g_{i} (\overline{x}), x 
  - \overline{x} \rangle \leq \overline{u}_{i} e_{i} (x, \overline{x}), ~ \forall 
  ~ i \in I(\overline{x}), ~ \forall ~ x \in C. \label{eq:01}
 \end{align}
  Since $\overline{u}_{j} = 0$ for all $j \in K(\overline{x})$, we have
 \begin{align}
  \overline{u}_{j} \langle \nabla g_{j} (\overline{x}), x - \overline{x} \rangle 
  = 0, ~ \forall ~ j \in K(\overline{x}), ~ \forall ~ x \in C. \label{eq:02}
 \end{align}
 By adding both \eqref{eq:01} and \eqref{eq:02}, we obtain
  \begin{align}
   \sum^{m}_{i = 1} \overline{u}_{i} \left\langle \nabla g_{i} (\overline{x}), 
   x - \overline{x} \right\rangle \leq \sum^{m}_{i = 1} \overline{u}_{i} e_{i} 
   (x, \overline{x}), ~ \forall ~ x \in C. 
  \label{eq:03}
 \end{align}
 
 Now, it follows by assumption \eqref{KKT:point} that
 \begin{align}
  & \hspace{1.0cm} \left\langle \nabla f(\overline{x}) + \sum^{m}_{i=1} \overline{u}_{i} \nabla g_{i} (\overline{x}), x - \overline{x} \right\rangle \geq \sum^{m}_{i = 1} \overline{u}_{i} e_{i} (x, \overline{x}), \notag \\
  & \Longleftrightarrow \, \left\langle \nabla f(\overline{x}), x - \overline{x}
  \right\rangle \geq - \sum^{m}_{i=1} \overline{u}_{i}  \left\langle 
  \nabla g_{i} (\overline{x}), x - \overline{x} \right\rangle + \sum^{m}_{i = 1}
  \overline{u}_{i} e_{i} (x, \overline{x}), \notag \\
  & \hspace{2.5cm} \overset{\eqref{eq:03}}{\Longrightarrow} ~  \left\langle \nabla f(\overline{x}), x - \overline{x} \right\rangle \geq 0. \label{eq:04}
 \end{align}
 Since $f$ is pseudoconvex, it follows from \eqref{eq:04} that
 $$f(x) \geq f(\overline{x}), ~ \forall ~ x \in C,$$
 i.e., $\overline{x}$ is a global solution of problem \eqref{extended:problem} with 
 $h \equiv 0$.
\end{proof}

In the particular case when every $e_{i} \equiv 0$, we have:

\begin{corollary}\label{coro:suff} {\rm (\cite[Theorem in pages 151-152]{Manga})}
 Let $K$, $f$, $g_{i}$, $\overline{x} \in K$ and $C$ and $I(\overline{x})$ defined as above. Suppose in addition that $C$ is convex, $f$ is pseudoconvex and every $g_{i}$ is quasiconvex (thus every $e_{i} \equiv 0$) for every $i \in I$. If there exist $\overline{u} \in \mathbb{R}^{m}$ such that $(\overline{x}, \overline{u})$ satisfies condition \eqref{KKT:point} (with every $e_{i} \equiv 0$), then $\overline{x}$ is a global solution of problem \eqref{extended:problem} with 
 $h \equiv 0$.
\end{corollary}

We conclude this paper with the following example.

\begin{example}
 Let $K := \, ]-\pi, \pi[$ and $f,g_{1}: K \rightarrow \mathbb{R}$ be given by $f(x) = - x$ and $g_{1} (x) = \sin(x)$. Then, $C := \left\{x \in K: \, \sin(x) \leq 0 \right\} = \, ]-\pi,0]$.
 We consider the problem
$$ \min_{x\in C} \, -x.$$
 Clearly, ${\rm argmin}_{C}\,f = \{0\}$. Furthermore, note that $f$ is convex (thus pseudoconvex) while $g_{1}$ is not quasiconvex on $K$, but $g_{1}$ is $e$-quasiconvex on $K$ with $e (x,y)= 2 |x-y|$ by \cite[Proposition 18]{Ali 2021}. As a consequence, the sufficient KKT results from \cite[Chapter 5]{J-2020} as well as the ones from \cite[Chapters 10 and 11]{Manga} can not be applied.
 
On the other hand, note that $\nabla f(\overline{x}) = -1$, $\nabla g(\overline{x}) = \cos(0) = 1$. Since $g_{1}$ is $e$-quasiconvex, it follows from \eqref{KKT:point} that
\begin{equation}\label{KKT1ex}
 (-1 + u_{1}) x \geq 2 u_{1}|x|, ~ \forall ~ x \in C = \, ]-\pi, 0], \notag
\end{equation}
then $u_{1} \leq \frac{1}{3}$. Hence, it follows by Theorem \ref{theo:suff} that $\overline{x} = 0$ is a global solution for $f$ on $C$, as expected.
\end{example}

\section{Conclusions}

We studied a new class of generalized quasiconvex functions, termed 
$e$-quasi\-con\-ve\-xi\-ty, which encompasses various important classes of 
functions including convex, Lipschitz continuous, $e$-convex, and quasiconvex functions. 
By ex\-ten\-ding key theoretical properties, we established existence results for the 
constraint and unconstraint minimization problems and finer sufficient KKT optimality
conditions for the nonconvex mathematical programming problem.

We hope these contributions will shed new light on further discussions of generalized convexity, generalized monotonicity, and their applications in optimization and 
variational inequalities.

\section{Declarations}

\subsection{Availability of supporting data}

No data sets were generated during the current study. 

\subsection{Competing Interests}

There are no conflicts of interest or competing interests related to this
manuscript.

\subsection{Funding}

 This research was partially supported by ANID--Chile through Fondecyt Regular 1241040 (Lara).

\subsection{Acknowledgments}

Part of this work was carried out when M.H. Alizadeh was visiting the Instituto de Alta Investigaci\'on (IAI) of the Universidad de Tarapac\'a (UTA), in Arica, Chile, during November-December 2024, this author wishes to thank the institute for its hospitality. This research was partially supported by the Iran National Science Foundation (INSF) under Project Number 40400113 (Alizadeh) and by ANID--Chile under project Fondecyt Regular 1241040 (Lara).


\begin{thebibliography}{99}

\bibitem {Ali 2018} 
 {\sc M.H. Alizadeh}, On generalized convex functions and generalized subdi\-ffe\-rential, {\it Optim. Lett.}, \textbf{14}, 157--169, (2020).

\bibitem{Ali 2021} 
 {\sc M.H. Alizadeh}, On generalized convex functions and generalized sub\-di\-ffe\-rential II, {\it Optim. Lett.}, \textbf{15}, 2225--2240, (2021).

\bibitem {Ali 203e-conv} 
 {\sc M.H. Alizadeh}, On $e$-convex functions and $e$-subdifferentials in locally convex spaces, {\it Optimization}, \textbf{72}, 1347--1362, (2023).

\bibitem {ABP} 
 {\sc M.H. Alizadeh, M. Bianchi, R. Pini}, On $e$-monotonicity and maxi\-ma\-li\-ty of operators in Banach spaces, {\it J. Global Optim.}, \textbf{91}, 155--170, (2025).


\bibitem {Ali-abadi} 
 {\sc M.H. Alizadeh, J. Hosseinabadi}, On $\sigma $-subdifferential polarity and Fr\'echet $\sigma$-subdifferential, {\it Numer. Funct. Anal. Optim.}, \textbf{44,} 603--618, (2023).

\bibitem {ALi-Mako} 
 {\sc M.H. Alizadeh, J. Mako}, On $e$-convexity, {\it Miskolc Math. Notes}, \textbf{23}, 51--60, (2022).

\bibitem {AZ 2023} 
 {\sc M.H. Alizadeh, A. Youhannaee Zanjani}, On the sum rules and maximality of generalized subdifferentials, {\it Optimization},
 \textbf{73}, 2139--2158, (2023).

\bibitem {AZ 2025} 
 {\sc M.H. Alizadeh, A. Youhannaee Zanjani}, On e-convex functions and their applications in optimality conditions, {\it Optimization}, 1-19
 (2025)

\bibitem {Ansari-book} 
 {\sc Q.H. Ansari, C.S. Lalitha, M. Mehta}. ``Generalized Convexity, Non\-smooth Variational Inequalities, and Nonsmooth Optimization". CRC Press, Cornwall, (2013).

\bibitem{AE}
 \textsc{K.J. Arrow, A.C. Enthoven}, Quasiconcave programming, \textit{Econo\-me\-tri\-ca}, \textbf{29}, 779--800, (1961).

\bibitem{ABM}
 {\sc H. Attouch, G. Buttazzo, G. Michaille}. ``Variational Analysis in Sobolev and BV Spaces: Aplications to PDEs and Optimization''. MPS-SIAM, Philadelphia, (2006). 

\bibitem{Aus2}
 {\sc A. Auslender}, How to deal with the unbounded in optimization: theory and algorithms, {\it Math. Program.}, {\bf 79}, 3--18, (1997). 

\bibitem {AT}
 {\sc A. Auslender, M. Teboulle}. ``Asymptotic Cones and Functions in Optimization and Variational Inequalities". Springer-Verlag, New York, Berlin, (2003).

\bibitem {ACJ 94} 
 {\sc D. Aussel, J.-N. Corvellec, M. Lassonde}, Subdifferential characterization of quasiconvexity and convexity, {\it J. Convex Anal.}, \textbf{1}, 195--201, (1994).

\bibitem{ADSZ}
 {\sc M. Avriel, W.E. Diewert, S. Schaible, I. Zang}. ``Generalized Concavity''. SIAM, Philadelphia, (2010).

\bibitem{BBGT}
 {\sc C. Baiocchi, G. Buttazzo, F. Gastaldi, F. Tomarelli}, General existence theorems for unilateral problems in continuum mechanics,  {\it Arch, J. Rational Mech. Anal.}, {\bf 100}, 149--189 (1988). 


\bibitem{CMA} 
 \textsc{A. Cambini, L. Martein}. ``Generalized Convexity and Optimization''. Springer-Verlag, Berlin-Heidelberg, (2009).

\bibitem{D-1959}
 \textsc{G. Debreu}. ``Theory of value". John Wiley, New York, (1959).

\bibitem{FFB-Vera}
 {\sc F. Flores-Baz\'an, F. Flores-Baz\'an, C.  Vera}, Maximizing and minimizing quasiconvex functions: related properties, existence and optimality conditions via radial epiderivates, {\it J. Global Optim.}, {\bf 63}, 99--123, (2015).

\bibitem{flo-had-lara-2}
 {\sc F. Flores-Baz\'an, N. Hadjisavvas, F. Lara, I. Montenegro}, First- and second-order 
asymptotic analysis with applications in quasiconvex optimization, {\it J. Optim. Theory 
Appl.}, {\bf 170}, 372--393, (2016).

\bibitem{HKS}
 {\sc N. Hadjisavvas, S. Komlosi, S. Schaible}. ``Handbook of Gene\-ra\-li\-zed Convexity and Generalized Monotonicity''. Springer-Verlag, Boston, (2005).

\bibitem{HLM}
 {\sc N. Hadjisavvas, F. Lara, J.E. Mart\'inez-Legaz}, A quasiconvex asymptotic function with applications in optimization, {\it J. Optim. Theory Appl.}, {\bf 180}, 170--186, (2019). 

\bibitem{HLL}
{\sc N. Hadjisavvas, F. Lara, D.T. Luc}, A general asymptotic function with applications in nonconvex optimization. {\it J. Global Optim.}, {\bf 78}, 49--68, (2020). 

\bibitem {Huang 24} { \sc H. Huang. H. Fang}, Rho-subdifferential and its application to optimization problems, {\it J. Ind. Manag. Optim.}, {\bf 21}, 4695--4711, (2025).
 
\bibitem {Hui-Sun}
 {\sc H. Huang, C. Sun}, Sigma-subdifferential and its application to minimization problem, {\it Positivity}, \textbf{24}, 539--551, (2020).

\bibitem{Iusem-Lara-5}
{\sc A. Iusem, F. Lara,} Quasiconvex optimization and asymptotic analysis in Banach spaces. {\it Optimization}, {\bf 69}, 2453--2470, (2020). 

\bibitem{J-2020}
 {\sc J. Jahn}. ``Introduction to the Theory of Nonlinear Optimization". Springer Nature, 4th Edition, (2020).

\bibitem{Jo-Luc-MI} 
 {\sc A. Jofr\'e, D.T. Luc, M. Th\'era}, $\epsilon$-subdifferential and $\epsilon$-monotonicity, {\it Nonlinear Anal.}, \textbf{33}, 71--90, (1998).


\bibitem{lara}
 {\sc F. Lara}, Generalized asymptotic functions in nonconvex multiobjective 
 optimization problems, {\it Optimization}, {\bf 66}, 1259--1272, (2017).

\bibitem{Lara-9}
 {\sc F. Lara}, On strongly quasiconvex functions: existence results and 
 proximal point algorithms, {\it J. Optim. Theory Appl.}, {\bf 192}, 891--911, (2022).

\bibitem{lara-lopez}
 {\sc F. Lara, R. L\'opez}, Formulas for asymptotic functions via conjugates, directional derivatives and subdifferentials, {\it J. Optim. Theory Appl.}, {\bf 173}, 793--811, (2017).


\bibitem{M1}
 {\sc O.L. Mangasarian}, Pseudo-convex functions, {\it J. SIAM Control, Ser. A}, {\bf 3},
 281--290, (1965).

\bibitem{Manga}
 {\sc O.L. Mangasarian}. ``Nonlinear Programming''. SIAM, Classics in A\-pplied
 Mathematics, Philadelphia, (1994). 

\bibitem{Penot} 
{\sc J. P. Penot}, What is quasiconvex analysis ?, {\it Optimization}, {\bf 47}, 35--110, (2000). 


\bibitem{rock} 
 \textsc{R.~T. Rockafellar}. ``Convex Analysis". Princeton University Press, Princeton, New Jersey, (1970).

\bibitem{Soltan}
 {\sc V. Soltan}, Asymptotic planes and closedness conditions for linear images and vector sums of sets, {\it J. of Convex Anal.}, {\bf 25}, 1183--1196, (2018). 

\end{thebibliography}
\end{document}